\documentclass[12pt]{amsart}
\usepackage[utf8]{inputenc}
\usepackage{amssymb}
\usepackage{graphicx}
\usepackage{graphicx,color}
\usepackage{latexsym}
\usepackage[all]{xy}
\usepackage{mathrsfs}
\usepackage{enumerate}
\usepackage{amssymb}
\usepackage{tikz}
\usepackage{breqn}

\makeatletter
\renewcommand{\subsection}{\@startsection
{subsection}{2}{0mm}{\baselineskip}{-0.25cm}
{\normalfont\normalsize\em}}
\makeatother

\def\hH{\widehat{H}}
\def\hS{\widehat{S}}
\def\hGamma{\widehat{\Gamma}}
\def\hLambda{\widehat{\Lambda}}

\def\negP{\mathbf P}
\def\negQ{\mathbf Q}
\def\negalpha{\text{\boldmath$\alpha$}}

\def\neg1{\text{\boldmath$1$}}
\def\negbeta{\text{\boldmath$\beta$}}
\def\neglambda{\text{\boldmath$\lambda$}}
\def\neggamma{\text{\boldmath$\gamma$}}
\def\negeta{\text{\boldmath$\eta$}}
\def\negdelta{\text{\boldmath$\delta$}}
\def\neggamma{\text{\boldmath$\gamma$}}

\def\negeta{\text{\boldmath$\eta$}}

\def\neg1{\text{\boldmath$1$}}

\def\hH{\widehat{H}}
\def\cC{\mathcal C}

\def\cX{\mathcal X}

\def\NN{\mathbb{N}}

\def\ZZ{\mathbb{Z}}

\def\FF{\mathbb{F}}

\DeclareMathOperator{\lub}{lub}

\newtheorem{theorem}{Theorem}[section]
\newtheorem{proposition}[theorem]{Proposition}
\newtheorem{corollary}[theorem]{Corollary}
\newtheorem{lemma}[theorem]{Lemma}

{\theoremstyle{definition}
\newtheorem{definition}[theorem]{Definition}

}

{\theoremstyle{remark}
\newtheorem{remark}[theorem]{Remark}
}

\title[]{Generalized Weierstrass semigroups at several points on certain maximal curves which cannot be covered by the Hermitian curve}
\author{M. Montanucci and G. Tizziotti}
\date{\today}

\begin{document}

\begin{abstract}
In this paper we determine the generalized Weierstrass semigroup $ \widehat{H}(P_{\infty}, P_1, \ldots , P_{m})$, and consequently the Weierstrass semigroup $H(P_{\infty}, P_1, \ldots , P_{m})$, at $m+1$ points on the curves $\mathcal{X}_{a,b,n,s}$ and $\mathcal{Y}_{n,s}$. These curves has been introduced in Tafazolian et al. \cite{TTT} as new examples of curves which cannot be covered by the Hermitian curve.
\end{abstract}


\maketitle

\section{Introduction}
Let $\mathcal{X}$ be a nonsingular, projective, geometrically irreducible algebraic curve of positive genus $g$ defined over a finite field $\mathbb{F}_q$ with $q$ elements and let $\mathcal{X}(\mathbb{F}_q)$ be the set of its $\mathbb{F}_q$-rational points. The curve $\mathcal{X}$ is called $\mathbb{F}_q$-\textit{maximal} if its number of $\mathbb{F}_{q}$-rational point attains the Hasse-Weil upper bound, namely equals $2g\sqrt{q}+q+1$. Clearly, maximal curves can only exist over  fields whose cardinality is a perfect square. Apart for being of theoretical interest as extremal objects, maximal curves over finite fields has attracted a lot of attention in recent decades due to their applications to coding theory and cryptography. 
Maximal curves are indeed special for the structure of the so-called Weiestrass semigroup at one point, which is the main ingredient used in the literature to construct AG codes with good parameters.
In \cite{D}, Delgado proposed an interesting generalization of Weierstrass semigroups, where instead of one single point several points of the curve are considered simultaneously. His original approach was on curves over algebraically closed fields, but later Beelen and Tutas \cite{BT} studied such object, called \textit{generalized Weierstrass semigroup}, for curves defined over finite fields. Moyano-Fern\' andez, Ten\'orio and Torres \cite{MTT} presented importante properties of generalized Weierstrass semigroup at several points on a curve, which suggested again the use of maximal curves to obtain good AG codes. In \cite{D} Delgado introduced also the notion of maximality to generalized Weierstrass semigroups  (see Definition \ref{def:maximals}). Interesting connections between the concepts of maximal elements and generating sets for Weiestrass semigroups was presented in \cite{MTT}, and in \cite{TT2}, the authors used the notion of maximality in generalized Weierstrass semigroup to give a characterization of gaps and pure gaps of Weierstrass semigroups.

The most important and well-studied example of a maximal curve is the so-called Hermitian curve $\mathcal{H}_q$ defined over $\mathbb{F}_{q^2}$ by the affine equation $y^q + y = x^{q+1}$. 
A well-known reason is that for fixed $q$, the curve $\mathcal{H}_q$ has the largest possible genus $g(\mathcal{H}_q ) = q(q- 1)/2$ that an $\mathbb{F}_{q^2}$-maximal curve can have. A result commonly attributed to Serre, see \cite[Proposition 6]{Serre}, gives that any $\mathbb{F}_{q^2}$-rational curve which is covered by an $\mathbb{F}_{q^2}$-maximal curve is itself also $\mathbb{F}_{q^2}$-maximal. Therefore many maximal curves can be obtained by constructing subcovers of already known maximal curves, in particular subcovers of the Hermitian curve. 
For a while it was speculated in the research community that perhaps all maximal curves could be obtained as subcovers of the Hermitian curve, but it was shown by Giulietti and Korchm\'aros that this is not the case, see \cite{GK}.
Giulietti and Korchm\'aros constructed indeed a maximal curve over $\mathbb{F}_{q^6}$, nowadays referred to as GK curve, which cannot be covered by the Hermitian curve whenever $q > 2$. Garcia, G\" uneri, and Stichtenoth, in \cite{GGS}, presented a new family of maximal curves over $\mathbb{F}_{q^{2n}}$ ($n$ odd), known as GGS curves, which generalizes the GK curve and that are not Galois-covered by the Hermitian curve \cite{GMZ, DM}. Many applications of these curves in coding theory have been made in recent years, see e.g. \cite{zini}, \cite{BMZ}, \cite{CT_GK}, \cite{GKcodes}, \cite{HY3} and \cite{CT mGK}. 
Another generalization of the GK curve over $\mathbb{F}_{q^{2n}}$ ($n$ odd) has been introduced by Beelen and Montanucci in \cite{BM}, which is now known as BM curves. These curves are not Galois-covered by the Hermitian curve as well. Applications of of the BM curves to coding theory can be found in \cite{LV} and \cite{MP}.

Tafazolian, Teher\'an-Herrera, and Torres \cite{TTT}  presented two further examples of maximal curves, denoted by $\mathcal{X}_{a,b,n,s}$ and $\mathcal{Y}_{n,s}$, that cannot be covered by the Hermitian curve. These examples are again closely related to the GK curve. They are not generalizations as the GGS and BM curves, but instead subcovers of the GK curve. The curves $\mathcal{X}_{a,b,n,s}$ and $\mathcal{Y}_{n,s}$ can be considered as concrete models of the quotient curves computed by Fanali and Giulietti in \cite[Theorem 4.5]{FG} while for $s=1$ the curve $\mathcal{Y}_{n,1}$ is the GGS curve mentioned above.
Recently, Br\'as-Amor\'os and Castellanos \cite{BS} determined the Weierstrass semigroup at certain $m+1$ points on the curves $\mathcal{X}_{a,b,n,1}$ and $\mathcal{Y}_{n,1}$ (case $s=1$) and present some conditions to find certain pure gaps for certain specific semigroups, yielding to some AG codes with good parameters. It is so natural to ask: what happens if $s>1$?
In this paper we determine the generalized Weierstrass semigroup $ \widehat{H}(P_{\infty}, P_1, \ldots , P_{m})$, and consequently the Weierstrass semigroup $H(P_{\infty}, P_1, \ldots , P_{m})$, at $m+1$ points on the curves $\mathcal{X}_{a,b,n,s}$ and $\mathcal{Y}_{n,s}$ for arbitrary $s \geq 1$, generalizing the results in \cite{BS}. 

The paper is organized as follows. In Section 2 we fix the notation that will be used throughout the paper and recall some basic facts about generalized Weierstrass semigroups and their maximal elements. In the same section the curves $\mathcal{X}_{a,b,n,s}$ and $\mathcal{Y}_{n,s}$, together with their most important properties, are also presented.  Section 3 is devoted to determine the generalized Weierstrass semigroup is at several points on the maximal curves $\mathcal{X}_{a,b,n,s}$, while in Section 4 we present similar results for the curve $\mathcal{Y}_{n,s}$.

\section{Preliminaries}

Let $\mathbb{F}_q$ denote a finite field with $q$ elements. Let $\mathcal{X}$ be a nonsingular, projective, geometrically irreducible algebraic curve of positive genus $g$ defined over $\mathbb{F}_q$ and let $\mathbf{Q}=(Q_1,\ldots,Q_m)$ an $m$-tuple of pairwise distinct rational points on $\mathcal{X}$. Throughout the paper the following notation will be used.

$\bullet$ $\mathbb{F}_q(\mathcal{X})$: function field associated to $\mathcal{X}$.

$\bullet$ $\mbox{Div}(\mathcal{X})$: the set of divisors on $\mathcal{X}$.

$\bullet$ $(f)$, $(f)_{\infty}$ and $(f)_{0}$: principal divisor, divisor of poles and divisor of zeros of $f \in \mathbb{F}_{q}(\mathcal{X})^\times$, respectively.

$\bullet$ $\mathcal{L}(G)$: the Riemann-Roch space associated to $G \in \mbox{Div}(\mathcal{X})$, that is, the $\mathbb{F}_{q}$-vector space $\{ h \in \mathbb{F}_{q}(\mathcal{X})^\times \mbox{ : } (h) + G \geq 0 \} \cup \{ 0 \}$.

$\bullet$ $\ell(G)$: the dimension of $\mathcal{L}(G)$.

$\bullet$ $R_{\mathbf{Q}}$: the ring of functions in $\mathbb{F}_{q}(\mathcal{X})^\times$ having poles only on the set $\{Q_1,\ldots,Q_m\}$, that is, that are regular outside $\{Q_1,\ldots,Q_m\}$.

$\bullet$ $\mathbb{N}_{0}$: the set of nonnegative integers.


$\bullet$ For a fixed $\negalpha = (\alpha_1, \ldots, \alpha_m) \in \mathbb{Z}^m$, $D_{\negalpha}$ denotes the divisor $\alpha_1Q_1 + \ldots \alpha_m Q_m$.

$\bullet$ $v_{P}$: the discrete valuation in the function field $\mathbb{F}_q(\mathcal{X})$ at the point $P \in \mathcal{X}$.

\subsection{The curves $\mathcal{X}_{a,b,n,s}$} \label{subsection X}

Let $q=p^a$ a power of a prime number $p$ and let $n,b,s\geq 1$ be integers satisfying the following:

$\bullet$ $n$ is odd;

$\bullet$ $b$ is a divisor of $a$;

$\bullet$ $s$ is a divisor of $\dfrac{q^n + 1}{q+1}$.

For $c \in \mathbb{F}_{q^2}$ with $c^{q-1}=-1$, we define the curve $\mathcal{X}_{a,b,n,s}$ over $\FF_{q^{2n}}$ by the affine equations

\begin{equation} \label{Xabns}
\displaystyle c y^{q+1} =  t(x):= \sum_{i=0}^{a/b - 1} x^{p^{ib}} \mbox{ and }
y^{q^2}-y  = z^{M},
\end{equation}
where $M = \dfrac{q^n + 1}{s (q+1)}$.
The curve $\mathcal{X}_{a,b,n,s}$ is $\FF_{q^{2n}}$-maximal and it has genus $$g(\mathcal{X}_{a,b,n,s})=\dfrac{q^{n+2}-p^b q^n - sq^3 + q^2+(s-1)p^b}{2sp^b}.$$ Furthermore, $\mathcal{X}_{a,b,n,s}$ is a subcover of the GGS curve, see \cite{TTT}. 

We will denote an affine point $P=(\alpha, \beta,\gamma) \in \mathcal{X}_{a,b,n,s}(\mathbb{F}_{q^{2n}})$ by $P_{(\alpha, \beta , \gamma)}$. The unique common pole of the functions $x,y$ and $z$ will be denote by $P_{\infty}$. Thus, we have the following divisors:
\begin{equation} \label{div x}
(x - \alpha) = (q+1)MP_{(\alpha,0,0)} - (q+1)MP_{\infty} \mbox{ };    
\end{equation}

\begin{equation} \label{div y}
\displaystyle (y - \beta) = \sum_{i=1}^{q/p^b} MP_{(\alpha_i,\beta,0)} - \dfrac{q}{p^b}MP_{\infty}, \mbox{ with } t(\alpha_{i}) = \beta^{q+1} \mbox{ and } \beta \in \mathbb{F}_{q^2} \mbox{ } ;    
\end{equation}

\begin{equation} \label{div z}
\displaystyle (z) = \sum_{j=1}^{q^2}\sum_{i=1}^{q/p^b} P_{(\alpha_i,\beta_j,0)} - \dfrac{q^3}{p^b} P_{\infty}, \mbox{ with } \beta_{j} \in \mathbb{F}_{q^2} \mbox{ and } c \beta_{j}^{q+1} = t(\alpha_i), \mbox{ for all } i,j.   
\end{equation}

From \cite[Proposition 5.1]{TTT} we have that $H(P_{\infty}) = \langle \frac{q}{p^b}M, \frac{q^3}{p^b}, (q+1)M \rangle$, which is a telescopic semigroup. More details about the curve $\mathcal{X}_{a,b,n,s}$ can be found in \cite{TTT}.

 \begin{remark} \label{canonical}
Since telescopic semigroups are symmetric (see \cite[Lemma 6.5]{kirfel}), the Frobenius number of $H(P_{\infty})$ is $\ell_g=2g-1=\dfrac{q^{n+2}-p^b q^n - sq^3 + q^2+(s-1)p^b}{sp^b} -1 = \dfrac{1}{p^b} [(q^2-p^b)(q+1)M - q^3]$.   
 \end{remark}

   \begin{proposition} \cite[Proposition 1.6.2]{stichtenoth} \label{canonical 2}
A divisor $K$ is canonical if and only if $\deg (K) = 2g-2$ and $\ell(K) \geq g$.
 \end{proposition}
 
 \begin{lemma} \label{K canonico}
The divisor 
$$K=\left( \dfrac{q^{n+2}-p^b q^n - sq^3 + q^2+(s-1)p^b}{sp^b} -2 \right)P_{\infty} =  \left(\dfrac{1}{p^b} [(q^2-p^b)(q+1)M - q^3] - 1 \right)P_{\infty}.
$$ 
is a canonical divisor of $\mathcal{X}_{a,b,n,s}$.
\end{lemma}
\begin{proof}
First, note that $\deg(K)=2g-2$. Now, by Remark \ref{canonical} the Frobenius number of $H(P_{\infty})$ is $ 2g-1$ and thus $\ell(K)=g$. Then the result follows from Proposition \ref{canonical 2}.
\end{proof}

\subsection{The curves $\mathcal{Y}_{n,s}$}
Let $q$ be a prime power. Let $n \geq 3$ be an odd integer and $s$ be a divisor of $\frac{q^n+1}{q+1}$. The $\mathcal{Y}_{n,s}$ curve defined over $\FF_{q^{2n}}$ is the projective curve defined by the equations

\begin{equation} \label{Yns}
\left\{ \begin{array}{rcl}
x^q+x & = & y^{q+1} \\
y^{q^2}-y & = & z^{M}
\end{array}\right.,
\end{equation}
where $M = \dfrac{q^n + 1}{s (q+1)}$.
The curve $\mathcal{Y}_{n,s}$ is $\FF_{q^{2n}}$-maximal and it has genus $$g(\mathcal{Y}_{n,s})=\dfrac{q^{n+2}-q^n - sq^3 + q^2+s-1}{2s}.$$ We note that, for $s=1$, the curve $\mathcal{Y}_{n,1}$ is the GGS curve. So, we can say that this curve generalizes the GGS curve. 

Let $\mathcal{Y}_{n,s}(\mathbb{F}_{q^{2n}})$ be the set of $\mathbb{F}_{q^{2n}}$-rational points of $\mathcal{Y}_{n,s}$, and we will denote $P=(\alpha, \beta,\gamma) \in \mathcal{Y}_{n,s}(\mathbb{F}_{q^{2n}})$ by $P_{(\alpha, \beta , \gamma)}$. The unique common pole of the functions $x,y$ and $z$ will be denote by $P_{\infty}$. Thus, we have the following divisors:
\begin{equation} \label{div x Yns}
(x - \alpha) = (q+1)MP_{(\alpha,0,0)} - (q+1)MP_{\infty} \mbox{ ;}    
\end{equation}

\begin{equation} \label{div y Yns}
\displaystyle (y - \beta) = \sum_{i=1}^{q} MP_{(\alpha_i,\beta,0)} - qMP_{\infty}, \mbox{ with } \alpha_{i}^{q} + \alpha_i = \beta^{q+1} \mbox{ and } \beta \in \mathbb{F}_{q^2}\mbox{ ;}  
\end{equation}

\begin{equation} \label{div z Yns}
\displaystyle (z) = \sum_{j=1}^{q^2}\sum_{i=1}^{q} P_{(\alpha_i,\beta_j,0)} - q^3 P_{\infty}, \mbox{ with } \beta_{j} \in \mathbb{F}_{q^2} \mbox{ and } \beta_{j}^{q+1} = \alpha_{i}^{q} + \alpha_i \mbox{ for all } i=1,\ldots,q.    
\end{equation}


From \cite[Proposition 5.1]{TTT} we have that $H(P_{\infty}) = \langle qM, q^3, (q+1)M \rangle$ and it is a telescopic semigroup. More details about the curve $\mathcal{Y}_{n,s}$ can be found in \cite{TTT}.

Analogously to Lemma \ref{K canonico} one can prove that the divisor

\begin{equation} \label{K canonico Yns}
    K  = \left( \dfrac{q^{n+2}-q^n - sq^3 + q^2+s-1}{s} -2 \right)P_{\infty} = [ (q^2-1)(q+1)M - q^3 - 1]P_{\infty}
\end{equation}

is a canonical divisor of $\mathcal{Y}_{n,s}$. Since the proof would be very similar to that of  Lemma \ref{K canonico}, we will omit it.

\subsection{Weierstrass semigroups and gaps}\label{WS} As before, let $\mathcal{X}$ be a curve over $\mathbb{F}_q$ and let $\mathbf{Q} = (Q_{1}, \ldots , Q_{m})$ be an $m$-tuple of pairwise distinct rational points on $\mathcal{X}$. The Weierstrass semigroup $H(\mathbf{Q})$ of $\cX$ at $\mathbf{Q}$ is the set

$$H(\mathbf{Q}) := \left\{(a_{1}, \ldots, a_{m}) \in \mathbb{N}_{0}^ {m} \mbox{ : } \exists \ h \in \mathbb{F}_q(\mathcal{X})^\times \mbox{ with } (h)_{\infty} = \sum_{i=1}^ {m} a_{i}Q_{i} \right\}.$$

In \cite{gretchen1}, G. Matthews introduced the concept of \textit{minimal generating set} for $H(\mathbf{Q})$, denoted by $\Gamma(\mathbf{Q}) = \Gamma(Q_{1}, \ldots , Q_{m})$. The word ``generating" comes from the fact that $H(\mathbf{Q})$ can be reconstructed from $\Gamma(\mathbf{Q})$, as we describe in the following.

Let ${\bf u}_1,\ldots,{\bf u}_t\in \mathbb{N}_0^{m}$, where ${\bf u}_k = (u_{k_{1}}, \ldots , u_{k_{m}})$ for all $k=1,\ldots t$. The \emph{least upper bound} ($\lub$) of the vectors ${\bf u}_1,\ldots,{\bf u}_t$ is defined as $$\lub\{{\bf u}_1,\ldots,{\bf u}_t\}=(\max\{{ u_{1_1}},\ldots,{ u_{t_1}}\},\ldots, \max\{{ u_{1_m}},\ldots,{ u_{t_m}}\} )\in \mathbb{N}_{0}^{m}.$$

In \cite[Theorem 7]{gretchen1}, it is shown that, if $2\leq m \leq q$, then 

$$
H(\mathbf{Q})= \{ \lub\{{\bf u}_1,\ldots,{\bf u}_m\}\in \mathbb{N}_0^m \mbox{ ; }  {\bf u}_1,\ldots,{\bf u}_m \in \Gamma(\mathbf{Q}) \}.
$$

The elements in the finite complement ${G(\mathbf{Q}):=\NN^m_0\backslash H(\mathbf{Q})}$, are called \emph{Weierstrass gaps}, or simply \textit{gaps}, of $\cX$ at $\mathbf{Q}$ (or simply \emph{gaps}. A \emph{pure Weierstrass gap} (or simply \textit{pure gap}) is an $m$-tuple $\negalpha \in G(\mathbf{Q})$ such that $\ell(D_\negalpha) = \ell(D_\negalpha - Q_j)$ for  $j=1,\ldots,m$. The set of pure gaps of $\cX$ at $\mathbf{Q}$ will be denoted by $G_0(\mathbf{Q})$. As mentioned in the Introduction, Weierstrass semigroups and their gaps are an important tool in coding theory, see e.g. \cite{CT}, \cite{GarciaKimLax} and \cite{gretchen3}. For more details about Weierstrass semigroups at $m$ points see \cite{CK} and \cite{CT}.

\subsection{Generalized Weierstrass semigroups and maximal elements}\label{GWS} Given $\mathcal{X}$ and $\mathbf{Q} = (Q_{1}, \ldots , Q_{m})$ as before, we define the \emph{generalized Weierstrass semigroup} of $\cX$ at $\mathbf{Q}$ to be the set
\begin{equation}\label{generalized WS}
 \widehat{H} (\mathbf{Q}):=\{(-v_{Q_1}(h),\ldots,-v_{Q_m}(h))\in \mathbb{Z}^m \ : \ h\in R_{\mathbf{Q}}\backslash\{0\} \}.
 \end{equation}
For further references and recent applications of these structures see \cite{BT,MTT,TT}. It is worth mentioning that the generalized Weierstrass semigroup $\widehat{H} (\mathbf{Q})$ is a different set than the (classical) Weierstrass semigroup $H(\mathbf{Q})$ of $\cX$ at $\mathbf{Q}$ presented in the previous section. However, the classical and generalized Weierstrass semigroups are related by $ H(\mathbf{Q}) = \widehat{H}(\mathbf{Q}) \cap \NN_0^m$ provided that $q \geq m$;  see \cite[Proposition 6]{BT}. 

For ${\negalpha=(\alpha_1,\ldots,\alpha_m)\in \mathbb{Z}^m}$ and $i\in \{1,\ldots,m\}$, we let
$$\nabla_i^m(\negalpha):=\{\negbeta=(\beta_1,\ldots,\beta_m)\in \hH(\negQ) \ : \ \beta_i=\alpha_i \ \text{and }\beta_j\leq \alpha_j \ \text{for }j\neq i\}.$$
Such sets play an important role in the characterization of $\hH(\mathbf{Q})$ as the following proposition shows.

\begin{proposition}\cite[Proposition 2.1]{TT} \label{prop 2.1.3 da tese}
Let $\negalpha \in \mathbb{Z}^m$ and assume that $q \geq m$. One has
\begin{enumerate}[\rm (1)]
\item $\negalpha \in \widehat{H} (\mathbf{Q})$ if and only if $\ell(D_{\negalpha}) = \ell(D_{\negalpha} - Q_i)+1$, for all $i \in \{1,\ldots,m\}$;
\item $\nabla_{i}^{m}(\negalpha) = \emptyset$ if and only if $\ell(D_{\negalpha})=\ell(D_{\negalpha} - Q_i)$.
\end{enumerate}
\end{proposition}

Let $\negalpha=(\alpha_1,\ldots,\alpha_m)\in \mathbb{Z}^m$. For a nonempty subset $J\subsetneq \{1,\ldots,m\}$, define
$$\nabla_J(\negalpha):=\{\negbeta=(\beta_1,\ldots,\beta_m)\in \hH(\negQ) \ : \ \beta_j=\alpha_j \ \text{for }j\in J, \ \text{and }\beta_i<\alpha_i  \ \text{for }i\not\in J\}.$$

\begin{definition}\label{def:maximals} An element $\negalpha\in \widehat{H} (\mathbf{Q})$ is called \emph{absolute maximal} if ${\nabla_J(\negalpha)=\emptyset}$ for every $J\subsetneq \{1,\ldots,m\}$ with $\#J\geq 2$. If otherwise ${\nabla_J(\negalpha)\neq\emptyset}$ for every $J\subsetneq  \{1,\ldots,m\}$ with $\#J\geq 2$, we say that $\negalpha\in \hH(\negQ)$ is \emph{relative maximal}.
\end{definition}

The sets of absolute and relative maximal elements in $\widehat{H} (\mathbf{Q})$ will be denoted, respectively, by $\hGamma(\mathbf{Q})$ and $\hLambda(\negQ)$.

\begin{proposition} \cite[Proposition 3.2]{TT} \label{absmax} Let $\negalpha\in \hH(\mathbf{Q})$ and assume that $q \geq m$. The following statements are equivalent:
\begin{enumerate}[\rm (i)]
\item $\negalpha\in \hGamma(\mathbf{Q})$;
\item $\nabla_i^m(\negalpha)=\{\negalpha\}$ for all $i\in \{1,\ldots,m\}$;
\item  $\nabla_i^m(\negalpha)=\{\negalpha\}$ for some $i\in \{1,\ldots, m\}$;
\item $\ell(D_{\negalpha})=\ell(D_{\negalpha} - \sum_{i=1}^m Q_i)+1$.
\end{enumerate}
\end{proposition}

Given a finite subset $\mathcal{B}\subseteq \ZZ^m$,  we define the \emph{least upper bound} ($\lub$) of $\mathcal{B}$ by
$$\mbox{lub}(\mathcal{B}):=(\max\{\beta_1\mbox{ : } \negbeta \in \mathcal{B} \},\ldots,\max\{\beta_m\mbox{ : } \negbeta \in \mathcal{B} \})\in \ZZ^m.$$

The next theorem shows that the set $\hGamma(\mathbf{Q})$ determines $\hH(\mathbf{Q})$ in terms of least upper bounds.

\begin{theorem}\cite[Theorem 3.4]{TT} Assume that $q \geq m$. The generalized Weierstrass semigroup of $\cX$ at $\mathbf{Q}$ can be written as
$$\hH(\mathbf{Q})=\{\lub(\{\negbeta^1,\ldots,\negbeta^m\}) \ : \ \negbeta^1,\ldots,\negbeta^m\in \hGamma(\mathbf{Q})\}.$$
\end{theorem}

In this way, $\hGamma(\mathbf{Q})$ can be seen as a \textit{generating set} of $\hH(\mathbf{Q})$ in the sense of \cite{gretchen1}. We observe that, unlike the case of generating set for classical Weierstrass semigroups of \cite{gretchen1}, the set $\hGamma(\mathbf{Q})$ is not finite. Nevertheless, it is finitely determined, in the sense that we describe in the following.\\

For $i=2,\ldots,m$, let $a_i$ be the smallest positive integer $t$ such that ${tQ_{i}-tQ_{i-1}}$ is a principal divisor on $\cX$. We can thus define the region
$$
\cC_m=\cC_m(\negQ):=\{\negalpha=(\alpha_1, \ldots , \alpha_m)\in \ZZ^m \ : \ 0\leq \alpha_i< a_i \ \mbox{for } i=2,\ldots,m\}.
$$
Let $\negeta^i = (\eta_1^i, \ldots, \eta_m^i)\in \ZZ^m$ be the $m$-tuple whose $j$-th coordinate is
$$
\eta^i_j=\left \lbrace \begin{array}{rcl}
-a_i & , & \mbox{if } j=i-1\\
a_i & , & \mbox{if } j=i\\
0 & , & \mbox{otherwise.}\\
\end{array}\right.
$$
Defining 
\begin{equation} \label{theta}
{\Theta_m=\Theta_m(\negQ):=\{ b_2 \negeta^2+ \ldots + b_{m} \negeta^{m}\in \ZZ^m \ : b_i \in \ZZ \ \mbox{for } i=2,\ldots,m\}},
\end{equation}
 we can state the following result concerning the absolute maximal elements and relative maximal elements in generalized Weierstrass semigroups at several points.

\begin{theorem}\cite[Theorem 3.7]{MTT}\label{maximals} Assume that $q \geq m$. The following holds:
$$\hGamma(\mathbf{Q})=(\hGamma(\mathbf{Q})\cap \cC_m(\negQ))+\Theta_m(\negQ)$$
$$\hLambda(\mathbf{Q})=(\hLambda(\mathbf{Q})\cap \cC_m(\negQ))+\Theta_m(\negQ).$$
\end{theorem}

\begin{remark} \label{obs gamma}
Note that since $\hGamma(\mathbf{Q})\cap \cC_m$ and $\hLambda(\mathbf{Q})\cap \cC_m(\negQ)$ are finite and $\Theta_m$ is finitely generated, $\hGamma(\mathbf{Q})$ and $\hLambda(\mathbf{Q})$ are determined by a finite number of elements in $\hH(\negQ)$. Also, note that $\Gamma(\mathbf{Q}) = \hGamma(\mathbf{Q}) \cap \mathbb{N}_{0}^{m}$ and so if we get $\hGamma(\mathbf{Q})$ we have $\Gamma(\mathbf{Q})$.
\end{remark}

For $\negalpha=(\alpha_1,\ldots,\alpha_m)\in \ZZ^m$, let $\overline{\nabla}_i(\negalpha):=\{\negbeta \in \ZZ^m \ : \ \beta_i=\alpha_i \mbox{ and } \beta_j<\alpha_j \ \mbox{for } j \neq i\}$ and $\overline{\nabla}(\negalpha):=\bigcup_{i=1}^m \overline{\nabla}_i(\negalpha)$, and define $\Lambda(\negQ):=\hLambda(\negQ)\cap \NN_0^m$. In \cite{TT2} the following characterization of gaps and pure gaps of a Weierstrass semigroup at several in terms of maximals elements in a generalized Weierstrass semigroup is provided.

\begin{theorem}  \cite[Theorems 3.1 and 3.2]{TT2} \label{teo gaps}
Assume that $q\geq m$. Then

1) $\displaystyle G(\negQ)=\bigcup_{\negbeta^*\in\Lambda(\negQ)} (\overline{\nabla}(\negbeta^*)\cap \NN_0^m),$ and

2) $\displaystyle G_0(\negQ)=\bigcup_{\scriptsize{(\negbeta^1,\ldots,\negbeta^m)\in\Lambda(\negQ)^m}}\left(\bigcap_{i=1}^m \overline{\nabla}_i(\negbeta^i)\right).$
\end{theorem}

An equivalent formulation for the absolute and relative maximal property can be made with respect to the concept of discrepancy, which has been introduced in \cite{DP}. This reformulation will be especially useful for us in calculations, when determining absolute and relative maximal elements.

\begin{definition}\label{def:discrepancy} 
Let $P$ and $Q$ be distinct rational points $P$ and $Q$ on $\mathcal{X}$. A divisor $A \in \mbox{Div}(\mathcal{X})$ is called a \textit{discrepancy} with respect to $P$ and $Q$ if $\mathcal{L}(A)\neq \mathcal{L}(A-Q)$ and $\mathcal{L}(A-P)=\mathcal{L}(A-P-Q)$.
\end{definition}

\begin{proposition} \cite[Proposition 2.8]{TT2} \label{relmax}
Let $\negalpha\in \ZZ^m$ and assume that $q \geq m$. The following statements are equivalent:
\begin{enumerate}[\rm (1)]
\item $\negalpha \in \hLambda(\negQ)$;
\item $\nabla(\negalpha)=\emptyset$ and $\ell(D_{\negalpha})=\ell(D_{\negalpha-\textbf{1}})+(m-1)$;
\item for $i,j\in I$ with $j\neq i$, $D_{\negalpha-\textbf{1}}+Q_i+Q_j$ is discrepancy w.r.t. $Q_i$ and $Q_j$;
\item there exists $i\in I$ such that $D_{\negalpha-\textbf{1}}+Q_i+Q_j$ is discrepancy w.r.t. $Q_i$ and $Q_j$, for any $j\neq i$;
\item there exists $i\in I$ such that $\nabla_i(\negalpha)=\emptyset$ and $\nabla_{i,j}(\negalpha)\neq \emptyset$ for every $j\neq i$.
\end{enumerate}
\end{proposition}

\begin{proposition} \cite[Proposition 3]{TT} \label{prop equiv discrepancia}
Let $\mathcal{X}$ be a curve over $\mathbb{F}_q$, $r$ be an integer such that $1 \leq r \leq q$ and $Q_1, \ldots , Q_r$ be rational points on $\mathcal{X}$. Let $\negalpha= (\alpha_1, \ldots , \alpha_{r}) \in \ZZ^r$ and $\mathbf{Q} = (Q_1, \ldots , Q_r)$. The following statements are equivalent:
\begin{enumerate}[\rm (i)]
\item $\negalpha \in \hGamma(\mathbf{Q})$;
\item $\alpha_1 Q_1 + \ldots + \alpha_r Q_r$ is a discrepancy with respect to any pair of distinct points in $\{Q_1,\ldots,Q_{r}\}$.
\end{enumerate}
\end{proposition}

We end this section presenting a technical lemma that will be an important tool in computations with discrepancies in the next sections.

\begin{lemma}\cite[Noether's Reduction Lemma]{fulton} \label{lemma noether}
Let $D$ be a divisor, $P \in \mathcal{X}$ and let $K$ be a canonical divisor. If $\dim(\mathcal{L}(D))>0$ and $\dim(\mathcal{L}(K-D-P)) \neq \dim(\mathcal{L}(K-D))$, then $\dim(\mathcal{L}(D+P))=\dim(\mathcal{L}(D))$.
\end{lemma}

\section{Generalized Weierstrass Semigroup at certain $m+1$ points on $\mathcal{X}_{a,b,n,s}$}

In this section we present our main results about Weierstrass semigroups and generalized Weierstrass semigroups at certain $m+1$ points on the curve $\mathcal{X}_{a,b,n,s}$ over $\mathbb{F}_{q^{2n}}$, where $1 \leq m \leq q/p^b$. Let $P_{\infty}$ and $P_{(\alpha,\beta,0)}$ on $\mathcal{X}_{a,b,n,s}$ as in Subsection \ref{subsection X}. To simplify the notation, we will denote $P_{(\alpha_i,0,0)}$ by $P_{i}$, for each $i=1,\ldots, q/p^b$, and $\mathbf{P}_{m+1}=(P_{\infty}, P_1, \ldots , P_m)$. Define $\mathcal{D}:=(z)_0 = \sum_{j=1}^{q^2}\sum_{i=1}^{q/p^b} P_{(\alpha_i,\beta_j,0)}$ as already introduced in Equation (\ref{div z}). 

We observe that with the conditions on $a,b$ and $n$ given in Section \ref{subsection X}, if $m \leq q/p^b$, then $m \leq q^{2n}$. This allows us to use the results given in Sections \ref{WS} and \ref{GWS} where this hypothesis was required.

\begin{proposition} \label{discrepancia m upla}
Let $(i,j)\in ([0,q]\times[1,M]\cap \ZZ^2) \backslash (q,M)$ and $1 \leq m \leq q/p^b$. Then the divisor $\displaystyle A= \dfrac{1}{p^b}[(q^2 - mp^b)(q+1)M -iqM -jq^3]P_{\infty} + \sum_{\ell=1}^{m}(iM+j)P_{\ell}$ is a discrepancy to every pair of distinct points in $\{P_{\infty}, P_1, \ldots,P_{m}\}$.
\end{proposition}

\begin{proof}
First, to see that $\mathcal{L}(A) \neq \mathcal{L}(A-P)$ for all $P \in \{ P_{\infty},P_1, \ldots,P_{m} \}$ we just note that the principal divisor of the function $f:=\dfrac{z^{M-j}y^{q-i}}{\prod_{k=1}^{m} (x-\alpha_k)}$ is
$$-\dfrac{1}{p^b}[(q^2 - mp^b)(q+1)M -iqM -jq^3]P_{\infty}-\sum_{\ell=1}^{m} (iM+j)P_{\ell} + \sum_{\ell=m+1}^{q/p^b} [(q+1)M-iM-j]P_{\ell} + (M-j) \widetilde{\mathcal{D}},$$
where $\widetilde{\mathcal{D}} = \mathcal{D} - (P_1 + \cdots + P_{q/p^b})$.

Hence $f \in \mathcal{L}(A) \setminus \mathcal{L}(A-P)$ for all $P \in \{ P_{\infty},P_1, \ldots,P_{m} \}$.

Now, we must prove that $\mathcal{L}(A-P) = \mathcal{L}(A-P-Q)$ and $\mathcal{L}(A-Q) = \mathcal{L}(A-Q-P)$, for all $P,Q\in\{P_{\infty},P_1,\ldots,P_{m}\}$. By Lemma \ref{lemma noether}, it suffices to prove that $\mathcal{L}(K-A+P)\neq\mathcal{L}(K-A+P+Q)$, where $K$ is a canonical divisor. Let $$K=\left(\dfrac{1}{p^b} [(q^2-p^b)(q+1)M - q^3] - 1 \right)P_{\infty}$$ be the canonical divisor given in Lemma \ref{K canonico}. Then, we have that
\begin{small}
$$
\begin{array}{ll}
K+P+Q-A& \displaystyle = \left(\dfrac{1}{p^b} [(q^2-p^b)(q+1)M - q^3] - 1 \right)P_{\infty} +P+Q \\
 & \\
& \displaystyle - \dfrac{1}{p^b}[(q^2 - mp^b)(q+1)M -iqM -jq^3]P_{\infty} - \sum_{\ell=1}^{m}(iM+j)P_{\ell} \\
         & \displaystyle = [(m-1)(q+1)M + i\dfrac{q}{p^b}M + (j-1)\dfrac{q^3}{p^b} - 1] P_\infty + P + Q- \sum_{\ell=1}^{m}(iM+j)P_{\ell}.
\end{array}
$$
\end{small}

If $P_{\infty} \in \{ P, Q \}$ then we can assume without loss of generality that $P=P_{\infty}$ and $Q=P_1$. Thus,
$$
K+P+Q-A \displaystyle = [(m-1)(q+1)M + i\dfrac{q}{p^b}M + (j-1)\dfrac{q^3}{p^b}] P_\infty  - (iM+j-1)P_1 - \sum_{\ell=2}^{m}(iM+j)P_{\ell}
$$
 and we have that $z^{j-1} y^i (x-\alpha_2) \cdots (x-\alpha_{m}) \in \mathcal{L}(K+P+Q-A)\setminus \mathcal{L}(K+Q-A)$.

If $P_{\infty} \not \in \{ P, Q \}$, we can suppose that $P=P_1$ and $Q=P_2.$ In this case, we have
$$
K+P+Q-A \displaystyle = [(m-1)(q+1)M + i\dfrac{q}{p^b}M + (j-1)\dfrac{q^3}{p^b} - 1] P_\infty - (iM+j-1)P_1  - (iM+j-1)P_2 - \sum_{\ell=3}^{m}(iM+j)P_{\ell}
$$
and $z^{j-1} y^i (x-\alpha_3) \cdots (x-\alpha_{m}) \in \mathcal{L}(K+P+Q-A)\setminus \mathcal{L}(K+Q-A)$.

So, $\mathcal{L}(K+P+Q-A) \neq \mathcal{L}(K+Q-A)$ and, by Lemma \ref{lemma noether}, we can conclude that $\mathcal{L}(A-P) = \mathcal{L}(A-P-Q).$ 

Therefore, $A$ is a discrepancy with respect to $P$ and $Q$ for any two distinct points $P,Q\in\{P_{\infty},P_1,\ldots,P_{m}\}$.
\end{proof}

Our next aim is to compute explicitly the absolute and relative maximal elements in generalized Weierstrass semigroups of $\mathcal{X}_{a,b,n,s}$ at $(m+1)$-tuples of rational points, where $1\leq m\leq q/p^b$. To do so, we observe that from Section \ref{subsection X} and Section \ref{GWS} one has
\begin{equation} \label{Cm}
\cC_{m+1}=\{\negbeta\in \ZZ^{m+1} \ : \ 0\leq \beta_i< (q+1)M \ \mbox{for } i=2,\ldots,m+1\},
\end{equation}

and $\Theta_{m+1}$ is generated by the $(m+1)$-tuples
\begin{equation}\label{eta_i}
\negeta^i=(0,\ldots,0,-(q+1)M,\underbrace{(q+1)M}_{i\text{-th entry}},0,\ldots,0)\in \ZZ^{m+1} \quad \mbox{for} \quad i=2,\ldots,m+1.
\end{equation}

For $(i,j)\in ([0,q]\times[1,M]\cap \ZZ^2) \backslash (q,M)$ define

\begin{equation} \label{alphas}
\negalpha^{i,j,m}:=\left(\dfrac{1}{p^b}[(q^2 - mp^b)(q+1)M -iqM -jq^3],iM+j,\ldots,iM+j\right)\in \ZZ^{m+1}.
\end{equation}

Note that when writing $\negalpha^{i,j,m} = (\alpha^{i,j,m}_0, \alpha^{i,j,m}_1, \ldots , \alpha^{i,j,m}_{m}) $ one has  $\alpha^{i,j,m}_0=\alpha^{i,j,m-1}_0-(q+1)M$, for all $m \geq 2$.

\begin{lemma} \label{lemma ast}
Let $\negalpha^{i,j,m} = (\alpha^{i,j,m}_0, \alpha^{i,j,m}_1, \ldots , \alpha^{i,j,m}_{m}) $ be as above with $m \geq 2$. Then
$$
\nabla_{1}^{m}(\alpha^{i,j,m}_0,\alpha^{i,j,m}_1-1+(q+1)M,\ldots,\alpha^{i,j,m}_{m-1}-1+(q+1)M,\alpha^{i,j,m}_m -1)=\emptyset .
$$
\end{lemma}
\begin{proof}
Let $\bar{\negalpha} = (\alpha^{i,j,m}_0, \alpha^{i,j,m}_1+(q+1)M,\ldots,\alpha^{i,j,m}_{m-1}+(q+1)M,\alpha^{i,j,m}_m)$ and $D_{\bar{\negalpha}} = \alpha^{i,j,m}_0 P_{\infty} + (\alpha^{i,j,m}_1+(q+1)M) P_1 + \ldots + (\alpha^{i,j,m}_{m-1}+(q+1)M) P_{m-1} + \alpha^{i,j,m}_m P_{m}$. In view of  Proposition \ref{prop 2.1.3 da tese}(2), it is sufficient to prove that
$$
\ell\left(D_{\bar{\negalpha}}-(P_1 + \ldots + P_{m})\right)=\ell\left(D_{\bar{\negalpha}} - (P_{\infty} +P_1 + \ldots + P_{m}) \right).
$$ 
Equivalently, by Lemma \ref{lemma noether}, let us show that
$$\mathcal{L} (K-D_{\bar{\negalpha}} + P_1 + \ldots + P_{m})\neq \mathcal{L} (K-D_{\bar{\negalpha}} + P_{\infty} +P_1 + \ldots + P_{m})$$
where $K$ as before denotes the canonical divisor $K = \left(\dfrac{1}{p^b} [(q^2-p^b)(q+1)M - q^3] - 1 \right)P_{\infty}$ given in Lemma \ref{K canonico}. 

Hence
$$
\begin{array}{rl}
K-D_{\bar{\negalpha}} + P_{\infty} +P_1 + \ldots + P_{m} & = \left( (m-1)(q+1)M + i \dfrac{qM}{p^b} + (j-1)\dfrac{q^3}{p^b} \right)P_{\infty}\\
 & \displaystyle -\sum_{k=1}^{m-1}[iM + (j-1)+(q+1)M]P_k - (iM+j-1)P_m.
\end{array}
$$

Let $f:= z^{j-1} y^i (x-\alpha_1) \ldots (x-\alpha_{m-1})$. Thus
$$
\begin{array}{rl}
(f) & = (j-1)P_1 + \cdots + (j-1)P_{q/p^b} - (j-1)\dfrac{q^3}{p^b} P_{\infty}\\
    & + iMP_1 + \cdots + iMP_{q/p^b} - i\dfrac{q}{p^b}M P_{\infty}  \\
    & + (q+1)M P_1 + \cdots + (q+1)M P_{m-1} - (m-1)(q+1)MP_{\infty}\\
    & = \displaystyle \sum_{k=1}^{m-1} [iM + (j-1) + (q+1)M]P_{k} + (iM + j-1)P_m \\
    & + \displaystyle \sum_{k=m+1}^{q/p^b}(iM+j-1)P_k - \left( (m-1)(q+1)M + i \dfrac{qM}{p^b} + (j-1)\dfrac{q^3}{p^b} \right)P_{\infty}. 
\end{array}
$$

So, $f \in \mathcal{L} (K-D_{\bar{\negalpha}} + P_{\infty} +P_1 + \ldots + P_{m}) \setminus \mathcal{L} (K-D_{\bar{\negalpha}} + P_1 + \ldots + P_{m})$, and the proof is complete.
\end{proof}

We are now ready to give an explicit description for the absolute and the relative maximal elements of $\hH(\negP_{m+1})$  in the region $\cC_{m+1}$ given in Equation (\ref{Cm}).

\begin{theorem} Let $1\leq m\leq q/p^b$, $\negP_{m+1} = (P_{\infty}, P_1, \ldots , P_{m})$, and $\hS_{m+1}$ be as above. 
Then
$$\hGamma(\negP_{m+1})\cap \cC_{m+1}=\hS_{m+1} \cup \{{\bf 0}\},$$
where $\hS_{m+1}$ is the set of all $\negalpha^{i,j,m}$ as in (\ref{alphas}).
\end{theorem}
\begin{proof}

It follows from Propositions \ref{prop equiv discrepancia} and \ref{discrepancia m upla} that $\hS_{m+1} \subseteq \hGamma(\negP_{m+1})\cap \cC_{m+1}$. To prove the reverse inclusion we will use induction on $m$.

First, for $m=1$ if $(\lambda, \mu)\in \hGamma(\negP_2)\cap \cC_2$, then $\mu < (q+1)M$ and so $\mu = iM+j$, for some $(i,j)\in ([0,q]\times[1,M]\cap \ZZ^2) \backslash (q,M)$. Now, by Propositions \ref{prop equiv discrepancia} and \ref{discrepancia m upla}, we have that $\left(\dfrac{1}{p^b}[(q^2 - mp^b)(q+1)M -iqM -jq^3],iM+j\right) \in \hGamma(\negP_2)\cap \cC_2$ and the results follows from Proposition \ref{absmax}.

 Let now $m\geq 2$ and suppose that $\hGamma(\mathbf{P}_{k+1})\cap \cC_{k+1}=\hS_{k+1}$ for $k=1,\ldots,m-1$. Recall that, when writing $\negalpha^{i,j,m} = (\alpha^{i,j,m}_0, \alpha^{i,j,m}_1, \ldots , \alpha^{i,j,m}_{m}) $, we have  $\alpha^{i,j,m}_0=\alpha^{i,j,m-1}_0-(q+1)M$. Given ${\negalpha=(\alpha_0, \alpha_1,\ldots,\alpha_m)\in \hGamma(\mathbf{P}_{m+1})\cap \cC_{m+1}}$, let $D_{\negalpha}=\alpha_0 P_{\infty} + \alpha_1P_1 \ldots + \alpha_m P_{m}$ and let us consider ${\tilde{t}:=\min\{t\in \NN \ : \ \alpha_0+t(q+1)M\geq 0\}}$. Observe that, by Proposition \ref{absmax}, $\ell(D_{\negalpha})\geq 1$, and thus $|\negalpha| = \deg(D_{\negalpha}) \geq 0$. In addition, $\negalpha \in \cC_{m+1}$ and hence $\alpha_0\geq -\sum_{j=1}^m \alpha_j\geq -m((q+1)M-1)$. So $\alpha_0+m (q+1)M\geq 0$, which shows that $1\leq \tilde{t} \leq m$. We distinguish two cases.

$\bullet$ If $1\leq \tilde{t} \leq m-1$, we have $2\leq m+1-\tilde{t}\leq m$. Without loss of generality, assume that $\alpha_1=\min_{1\leq k\leq m}\{\alpha_k\}$. Hence Theorem \ref{maximals} and Equation \eqref{eta_i} lead to
\begin{equation*}
\negalpha'=(\alpha_0+\tilde{t}(q+1)M,\alpha_1,\ldots,\alpha_{m-\tilde{t}},\alpha_{m-\tilde{t}+1}-(q+1)M,\ldots,\alpha_m-(q+1)M)\in \hH(\mathbf{P}_{m+1}).
\end{equation*}
Since $0\leq \alpha_k< (q+1)M$ for $k=1,\ldots, m$, we get
$$\negbeta=\mbox{lub}(\negalpha',\textbf{0})=(\alpha_0+\tilde{t}(q+1)M,\alpha_1,\ldots,\alpha_{m-\tilde{t}},0,\ldots,0)\in \hH(\mathbf{P}_{m+1-\tilde{t}}).$$
In particular, $\nabla_2^{m+1-\tilde{t}}(\negbeta)\neq \emptyset$. From the induction hypothesis, there exists $\negalpha^{i,j,m+1-\tilde{t}}\in \hS_{m+1-\tilde{t}}$ such that $\negalpha^{i,j,m+1-\tilde{t}}\in\nabla_2^{m+1-\tilde{t}}(\negbeta)$. As $\alpha_0^{i,j,m}=\alpha_0^{i,j,m-\tilde{t}}-\tilde{t}(q+1)M$ and $\alpha_1=iM+j\leq \alpha_k$ for $k=2,\ldots,m$, it follows that $\negalpha^{i,j,m}\in \nabla_2^{m}(\negalpha)$. Therefore, $\negalpha=\negalpha^{i,j,m}$ by Proposition \ref{absmax}, since $\negalpha\in \hGamma(\mathbf{P}_m)$.

$\bullet$ If $\tilde{t}=m$, by using a similar argument, we obtain $\alpha_0+m(q+1)M\in H(P_{\infty})=\langle \frac{q}{p^b}M,\frac{q^3}{p^b},(q+1)M \rangle$.  As $\alpha_0+m(q+1)M< (q+1)M$, we have either $\alpha_0+m(q+1)M=0$ or $\alpha_0+m(q+1)M = \lambda \dfrac{qM}{p^b} + \eta \dfrac{q^3}{p^b}$, for some integers $\lambda, \eta \geq 0$. 

If $\alpha_0+m(q+1)M=0$ then $\alpha_0=-m(q+1)M$, which is a contradiction since $\alpha_0\geq -m((q+1)M-1)>-m(q+1)M$. 

In the other case, we can write $\alpha_0+m(q+1)M= (q-i) \dfrac{qM}{p^b} + (M-j)\dfrac{q^3}{p^b} =\dfrac{1}{p^b}[q^2(q+1)M -iqM -jq^3]$ for some $(i,j)\in ([0,q]\times[1,M]\cap \ZZ^2) \backslash (q,M)$. So, we have that $\alpha_0=\dfrac{1}{p^b}[(q^2 - mp^b)(q+1)M -iqM -jq^3]=\alpha^{i,j,m}_0$. Suppose that $\negalpha$ and $\negalpha^{i,j,m}$ are not comparable in the partial order $\leq$. Without loss of generality, we may assume that $\alpha_m<\alpha_m^{i,j,m}$.
In this way, since $\alpha^{i,j,m}_k \geq 1$ for all $k=1,\ldots,m-1$, we have that $\negalpha\in \nabla_1(\alpha^{i,j,m}_0,\alpha^{i,j,m}_1 - 1+(q+1)M,\ldots,\alpha^{i,j,m}_{m-1}-1+(q+1)M,\alpha^{i,j,m}_m)$. However, by Lemma \ref{lemma ast}, we have
$$ \qquad \nabla_1(\alpha^{i,j,m}_0,\alpha^{i,j,m}_1-1+(q+1)M,\ldots,\alpha^{i,j,m}_{m-1}-1+(q+1)M,\alpha^{i,j,m}_m)=\emptyset.$$
Thus $\negalpha$ and $\negalpha^{i,j,m}$ are comparable with respect to $\leq$, which leads to $\negalpha=\negalpha^{i,j,m}$ by their absolute maximality. 

Therefore, $\hGamma(\mathbf{P}_{m+1})\cap \cC_{m+1}\subseteq \hS_{m+1}$ and the proof is complete.
\end{proof}

For $i,j,k_2, \ldots,k_{m+1} \in \mathbb{Z}$, define
\begin{equation} \label{beta classico}
\begin{array}{ll}
\neggamma_{k_2,\ldots,k_{m+1}}^{i,j,m}:= & \left(\dfrac{1}{p^b}[(q^2 - mp^b - p^b(\sum_{\ell=2}^{m+1} k_\ell))(q+1)M -iqM -jq^3],\right . \\
& \left . k_2(q+1)M + iM+j,\ldots,k_{m+1}(q+1)M + iM+j \right) \in \ZZ^{m+1}.
\end{array}
\end{equation}

The set of absolute maximal elements in the generalized Weierstrass semigroup $\hH(\negP_{m+1})$ is presented in the following corollary.

\begin{corollary} \label{corolario gamma}
$$
\hGamma(\negP_{m+1})= \{ \neggamma_{k_2,\ldots,k_{m+1}}^{i,j,m} \in \ZZ^{m+1} \mbox{ ; }  (i,j)\in ([0,q]\times[1,M]\cap \ZZ^2) \backslash (q,M) \mbox{ and }k_2, \ldots,k_{m+1} \in \mathbb{Z} \}.
$$
\end{corollary}

\begin{proof}
From Theorem \ref{maximals} we have that 
\begin{equation} \label{eq gamma m 1}
\hGamma(\negP_{m+1})=(\hGamma(\negP_{m+1}) \cap \cC_{m+1} )+\Theta_{m+1},
\end{equation}

where $ \Theta_{m+1}=\Theta_{m+1}(\negP_{m+1}) = \{ b_2 \negeta^2+ \ldots + b_{m} \negeta^{m}\in \ZZ^{m+1} \ : b_i \in \ZZ  \mbox{ for } i=2,\ldots,m+1\}$ and $\negeta^i=(0,\ldots,0,-(q+1)M,\underbrace{(q+1)M}_{i\text{-th entry}},0,\ldots,0)\in \ZZ^{m+1} \quad \mbox{for} \quad i=2,\ldots,m+1,$ as in (\ref{eta_i}). Hence
$$
\begin{array}{rl}
 \Theta_{m+1} = & \{ (-b_2 (q+1)M, (b_2-b_3)(q+1)M , (b_3-b_4)(q+1)M, \ldots ,\\
                           & (b_m-b_{m+1})(q+1)M, b_{m+1}(q+1)M) \in \ZZ^{m+1} \mbox{ ; } b_i \in \ZZ  \mbox{ for } i=2,\ldots,m+1 \}\\
                        = & \{ ( - \sum_{\ell = 2}^{m+1} k_\ell (q+1)M , k_2 (q+1)M, k_3(q+1)M, \ldots , k_{m+1}(q+1)M) \mbox{ ; } \\
                           &  k_\ell \in \ZZ  \mbox{ for } \ell=2,\ldots,m+1  \}.
 \end{array}
$$

By the previous theorem we have that $\hGamma(\negP_{m+1})\cap \cC_{m+1}=\hS_{m+1} \cup \{{\bf 0}\}$, where $\hS_{m+1}$ is the set of all $\negalpha^{i,j,m}$ as in (\ref{alphas}). So, the result follows from Equality (\ref{eq gamma m 1}) above.
\end{proof}

In the following corollary we determine the minimal generating set of the Weierstrass semigroup $H(\negP_{m+1})$. 

\begin{corollary} \label{classical WS}
Let $i,j,k_2, \ldots, k_{m+1}$ and $\neggamma_{k_2,\ldots,k_{m+1}}^{i,j,m}$ be as above. Then, $\Gamma(\negP_{m+1})$ is the set of all $\neggamma_{k_2,\ldots,k_{m+1}}^{i,j,m} \in \mathbb{N}_{0}^{m+1}$.
\end{corollary}

\begin{proof}
The result follows directly from Remark \ref{obs gamma} and Corollary \ref{corolario gamma}.
\end{proof}

\bigskip

\begin{theorem} Let $1\leq m\leq q/p^b$. Let
$$\negbeta^{0,0,m}=\left((m-1)(q+1)M,0,\ldots,0\right)\in \ZZ^{m+1}$$
and for $(i,j)\in ([0,q]\times[1,M]\cap \ZZ^2) \backslash (q,M)$, let
$$\negbeta^{i,j,m}:=\left(\dfrac{1}{p^b}[(q^2 - p^b)(q+1)M -iqM -jq^3],iM+j,\ldots,iM+j\right)\in \ZZ^{m+1}.$$
Then
$$\hLambda(\negP_{m+1})\cap \cC_{m+1}=\{\negbeta^{i,j,m} \ : (i,j)\in ([0,q]\times[1,M]\cap \ZZ^2) \backslash (q,M)\}\cup \{\negbeta^{0,0,m}\}.$$
\end{theorem}
\begin{proof} If $m=1$, then the result follows from the previous theorem. Therefore, we can suppose $m\geq 2$. Let $R_{m+1}=\{\negbeta^{i,j,m} \ : (i,j)\in ([0,q]\times[1,M]\cap \ZZ^2) \backslash (q,M)\}\cup \{\negbeta^{0,0,m}\}$. Let us prove that $R_{m+1}\subseteq \hLambda(\negP_{m+1})\cap \cC_{m+1}$. Note that $R_{m+1} \subseteq C_{m+1}$, since $(i,j)\in ([0,q]\times[1,M]$. To prove that $R_{m+1}\subseteq \hLambda(\negP_{m+1})$, by Proposition \ref{relmax} (3), it is sufficient to prove that the divisor $D=D_{\negbeta^{i,j,m}-\textbf{1}+\textbf{e}_1+\textbf{e}_k}$, with $(i,j)\in (([0,q]\times[1,M]\cap \ZZ^2) \backslash (q,M)) \cup \{ (0,0) \}$, is a discrepancy with respect to $P_{\infty}$ and $P_k$ for any $k=1,\ldots,m$.

First, suppose $(i,j)\neq (0,0)$. Thus 
$$
\displaystyle D=\dfrac{1}{p^b}[(q^2 - p^b)(q+1)M -iqM -jq^3]P_{\infty} + (iM+j)P_k + \sum_{\ell=1 \atop \ell \neq k}^{m} (iM+j-1)P_{\ell}.
$$

Let $\mathcal{D}:=\sum_{j=1}^{q^2}\sum_{i=1}^{q/p^b} P_{(\alpha_i,\beta_j,0)}$ (as above). Since, 
$$
\begin{array}{ll}
\left( \dfrac{z^{M-j} y^{q-i}}{x-\alpha_k} \right) = & -\dfrac{1}{p^b}[(q^2 - p^b)(q+1)M -iqM -jq^3]P_{\infty} - (iM+j)P_k\\
& \\
 & \displaystyle + (M-j) (\mathcal{D} - P_k) + (q-i)M \sum_{\ell=1 \atop \ell \neq k}^{m} P_{\ell},
\end{array}
$$
we have that $\dfrac{z^{M-j} y^{q-i}}{x-\alpha_k} \in \mathcal{L}(D) \setminus \mathcal{L}(D - P_k)$. 

Now, we must prove that $\mathcal{L}(D-P_{\infty}) = \mathcal{L}(D - P_{\infty} - P_{k})$. By Lemma \ref{lemma noether}, it suffices to prove that $\mathcal{L}(K-D+P_{\infty}+P_k) \neq \mathcal{L}(K-D + P_{\infty})$, where $K$ is the canonical divisor $K=\left(\dfrac{1}{p^b} [(q^2-p^b)(q+1)M - q^3] - 1 \right)P_{\infty}$. Note that
$$
K-D+P_{\infty}+P_k = \dfrac{1}{p^b}(iqM + (j-1)q^3)P_{\infty} - \sum_{\ell=1 }^{m} (iM+j-1)P_{\ell}.
$$
Thus, $z^{j-1}y^i \in \mathcal{L}(K-D+P_{\infty}+P_k) \setminus \mathcal{L}(K-D + P_{\infty})$, and it follows that $D$ is discrepancy with respect to $P_{\infty}$ and $P_k$ for any $k=1,\ldots,m$.

Therefore, to conclude that $R_{m+1}\subseteq \hLambda(\negP_{m+1})\cap \cC_{m+1}$, it remains to verify that $\negbeta^{0,0,m} \in \hLambda(\negP_{m+1})$. 

For $(i,j)=(0,0)$, we have $\displaystyle D= (m-1)(q+1)MP_{\infty}- \sum_{\ell=1 \atop \ell \neq k}^{m}P_{\ell}$. Again, we will prove that $\mathcal{L}(D)\neq \mathcal{L}(D-P_k)$ and $\mathcal{L}(K-D+P_{\infty})\neq\mathcal{L}(K-D+P_{\infty}+P_k)$.

Note that $\displaystyle \prod_{\ell=1 \atop \ell \neq k}^{m}(x - \alpha_{\ell}) \in \mathcal{L}(D) \setminus \mathcal{L}(D-P_{k})$. Moreover, since
$$
\displaystyle K-D+P_k+P_{\infty} = \dfrac{1}{p^b}[(q^2 - mp^b)(q+1)M -q^3]P_{\infty} + \sum_{\ell=1}^{m}P_{\ell},
$$
we get $\displaystyle \dfrac{z^{M-1} y^q}{(x-\alpha_1) \ldots (x-\alpha_{m})} \in \mathcal{L}(K-D+P_k+P_{\infty}) \setminus \mathcal{L}(K-D+P_{\infty})$.

Therefore, we conclude that $R_{m+1}\subseteq \hLambda(\negP_{m+1})\cap \cC_{m+1}$.

Now, let $\negbeta\in \hLambda(\negP_{m+1})\cap \cC_{m+1}$. Since $\negbeta\in \hLambda(\negP_{m+1})$, from definition we have that $\nabla_k^{m+1}(\negbeta)\neq \emptyset$ for any $k\in I$. So, there exists an absolute maximal element $\negalpha^{i,j,m}\in \nabla_2^{m+1}(\negbeta)$, where $\negalpha^{i,j,m}$ is given in (\ref{alphas}). Thus $\beta_2=iM+j$ and $\beta_3\geq iM+j$. Similarly, there exists an absolute maximal element $\negalpha^{i',j',m}\in \nabla_3^{m+1}(\negbeta)$, and thus $\beta_3=i'M+j'$ and $\beta_2\geq i'M+j'$. Hence $iM+j=i'M+j'$, and therefore $(i,j)=(i',j')$. Proceeding in the same way with pairs of the remaining indexes, we conclude that there exists an absolute maximal element $\negalpha^{i,j,m}\in \bigcap_{k=2}^{m+1} \nabla_k^{m+1}(\negbeta)$ and, in particular, we can conclude that $\beta_k=iM+j$ for $k=2,\ldots,m+1$. As $\negbeta\in\hLambda(\negP_{m+1})$, it follows that $\negbeta\neq \negalpha^{i,j,m}$ and thus $\beta_1>\alpha^{i,j,m}_1$. Hence, for each $\negbeta\in \hLambda(\negP_{m+1})\cap \cC(\negP_{m+1})$, there exists a unique $\negalpha^{i,j,m}\in \hGamma(\negP_{m+1})\cap \cC_{m+1}$ such that $\negalpha^{i,j,m}\in \nabla_{I\backslash\{1\}}(\negbeta)$. Therefore, $\#(\hGamma(\negP_{m+1})\cap \cC_{m+1})\geq \#(\hLambda(\negP_{m+1})\cap \cC_{m+1})$. As $\#R_{m+1}=\#(\hGamma(\negP_{m+1})\cap \cC_{m+1})$ and $R_{m+1}\subseteq \hLambda(\negP_{m+1})\cap \cC_{m+1}$, we have $ \hLambda(\negP_{m+1})\cap \cC_{m+1}=R_{m+1}$, which proves the result.
\end{proof}

Now, for $i,j,k_2,\ldots,k_{m+1} \in \mathbb{Z}$, define
$$
\begin{array}{ll}
\negdelta_{k_2,\ldots,k_{m+1}}^{i,j} := & \left(\dfrac{1}{p^b}[(q^2 - p^b(1+\displaystyle \sum_{\ell=2}^{m+1} k_\ell))(q+1)M -iqM -jq^3], \right. \\
& \left. k_2(q+1)M + iM+j,\ldots,k_{m+1}(q+1)M + iM+j \right) \in \ZZ^{m+1},
 \end{array}
$$
and 
$$
\neglambda_{k_2, \ldots , k_{m+1}}:= \left((m-1 - \sum_{\ell=2}^{m+1} k_\ell)(q+1)M,k_2(q+1)M,\ldots,k_{m+1}(q+1)M\right) \in \ZZ^{m+1}.
$$

Note that, if $(i,j)\in ([0,q]\times[1,M]$, then $k(q+1)M +iM+j \geq 0$ if and only if $k\geq 0$.

From the previous result and Theorem \ref{maximals} we get the set of relative maximal elements in the generalized Weierstrass semigroup $\hH(\negP_{m+1})$.

\begin{corollary} \label{corolario lambda chapeu}
Let $\negdelta_{k_2,\ldots,k_{m+1}}^{i,j}$ and $\neglambda_{k_2, \ldots , k_{m+1}}$ be as above.
Then
$$
\begin{array}{ll}
\hLambda(\negP_{m+1})= & \{  \negdelta_{k_2,\ldots,k_{m+1}}^{i,j} \in \ZZ^{m+1} \mbox{ ; }  (i,j)\in ([0,q]\times[1,M]\cap \ZZ^2) \backslash (q,M) \mbox{ , }k_2, \ldots, k_{m+1} \in \ZZ \}\\
& \bigcup \{ \neglambda_{k_2, \ldots , k_{m+1}}  \in \ZZ^{m+1} \mbox{ ; }  k_2, \ldots, k_{m+1} \in \ZZ \}.
\end{array}
$$
\end{corollary}

By Theorem \ref{teo gaps} we have that the gaps and pure gaps of $H(\negP_{m+1})$ can be obtained from elements in the set $\Lambda(\negP_{m+1})$. Since $\Lambda(\negP_{m+1})=\hLambda(\negP_{m+1})\cap \NN_0^m$, using the previous result we have the following.

\begin{corollary} \label{corolario lambda}
Let $\negdelta_{k_2,\ldots,k_{m+1}}^{i,j}$ and $\neglambda_{k_2, \ldots , k_{m+1}}$ be as above and let $\tau_{(i,j)}:= \lfloor (q^3(M-j) + qM(q-i) - p^b (q+1)M)/p^b (q+1)M \rfloor$. Then
$$
\begin{array}{ll}
\Lambda(\negP_{m+1})= &  \{  \negdelta_{k_2,\ldots,k_{m+1}}^{i,j} \in \hLambda(\negP_{m+1}) \mbox{ ; }  k_2, \ldots, k_{m+1} \in \mathbb{N}_{0} \mbox{ with }  \sum_{\ell=2}^{m+1} k_{\ell} \leq \tau_{(i,j)}\}\\
& \bigcup \{ \neglambda_{{k_2}, \ldots , k_{m+1}} \in \hLambda(\negP_{m+1})  \mbox{ ; }k_2,\ldots,k_{m+1} \in \mathbb{N}_{0} \mbox{ with } \sum_{\ell=2}^{m+1} k_{\ell} \leq m-1 \}.
\end{array}
$$
\end{corollary}
\begin{proof}
By definition, we have that $\Lambda(\negP_{m+1})=\hLambda(\negP_{m+1})\cap \NN_0^{m+1}$. Note that, $\negdelta_{k_2,\ldots,k_{m+1}}^{i,j} \in \NN_0^{m+1}$ if and only if $ k_2, \ldots, k_{m+1} \in \mathbb{N}_{0}^{m+1}$ and $ \sum_{\ell=2}^{m+1} k_{\ell} \leq \tau_{(i,j)}$. And $\neglambda_{k_2, \ldots , k_{m+1}} \in \mathbb{N}_{0}^{m+1}$ if and only if $\negdelta_{k_2,\ldots,k_{m+1}}^{i,j} \in \NN_0^m$ and  $\sum_{\ell=2}^{m+1} k_{\ell} \leq m-1$. So, the result follows from Corollary \ref{corolario lambda chapeu}. 
\end{proof}

\medskip

\begin{lemma}\label{delta lambda}
Let $k_2,\ldots,k_{m+1}, k_{2}',\ldots,k_{m+1}' \in \mathbb{N}_{0}$ and $(i,j), (i',j')\in ([0,q]\times[1,M]\cap \ZZ^2)\backslash (q,M)$. Then, we have that
\begin{enumerate}[\rm (1)]
\item $\negdelta_{k_2,\ldots,k_{m+1}}^{i,j} \neq \neglambda_{k_{2}', \ldots , k_{m+1}'}$; 

\item if $(k_2,\ldots,k_{m+1}) \neq (k_{2}',\ldots,k_{m+1}')$, then $\negdelta_{k_2,\ldots,k_{m+1}}^{i,j} \neq \negdelta_{k_{2}',\ldots,k_{m+1}'}^{i',j'}$, and $ \neglambda_{k_{2}, \ldots , k_{m+1}} \neq \neglambda_{{k_2}', \ldots , k_{m+1}'}$; and

\item if $(i,j) \neq (i',j')$, then $\negdelta_{k_2,\ldots,k_{m+1}}^{i,j} \neq \negdelta_{k_{2}',\ldots,k_{m+1}'}^{i',j'}$.
\end{enumerate}
\end{lemma}

\begin{proof}
$(1)$. Suppose that $\negdelta_{k_2,\ldots,k_{m+1}}^{i,j} = \neglambda_{k_{2}', \ldots , k_{m+1}'}$ for some $k_2,\ldots,k_{m+1}, k_{2}',\ldots,k_{m+1}' \in \mathbb{N}_{0}$. Then, $k_2(q + 1)M + iM + j = k_{2}'(q+1)M$, and we  have that $j = [(k_{2}' - k_2)(q + 1)-i]M$, a contradiction, since $(i,j) \in ([0,q]\times[1,M]\cap \ZZ^2)\backslash (q,M)$.

$(2)$. Let $(k_2,\ldots,k_{m+1}) \neq (k_{2}',\ldots,k_{m+1}')$. So, there is $t \in \{ 2,\ldots,m+1 \}$ such that $k_t \neq k_{t}'$.  Suppose that $\negdelta_{k_2,\ldots,k_{m+1}}^{i,j} = \negdelta_{k_{2}',\ldots,k_{m+1}'}^{i',j'}$. Then, $ k_t(q+1)M + iM+j =  k_{t}'(q+1)M + i'M+j'$. Thus, we have $j'- j = [(k_t - k_{t}')(q+1)+i-i']M$, a contradiction, since $k_t - k_{t}' \neq 0$ and $(i,j), (i',j')\in ([0,q]\times[1,M]\cap \ZZ^2)\backslash (q,M)$.

It is clear that $ \neglambda_{k_{2}, \ldots , k_{m+1}} \neq \neglambda_{{k_2}', \ldots , k_{m+1}'}$, since $k_t(q + 1)M \neq k_{t}'(q+1)M$.

$(3)$. Suppose that $\negdelta_{k_2,\ldots,k_{m+1}}^{i,j} = \negdelta_{k_{2}',\ldots,k_{m+1}'}^{i',j'}$. By previous item, we have that $k_t = k_{t}'$, for all $t \in \{ 2,\ldots,m+1 \}$. So, we get $j'- j = (i-i')M$. Now, since $j,j' \in [1,M]$, we have $j=j'$  and $i=i'$. Therefore, we conclude that if $(i,j) \neq (i',j')$, then $\negdelta_{k_2,\ldots,k_{m+1}}^{i,j} \neq \negdelta_{k_{2}',\ldots,k_{m+1}'}^{i',j'}$.
\end{proof}

\begin{proposition}
For $(i,j) \in ([0,q]\times[1,M]\cap \ZZ^2)\backslash (q,M)$, let $\tau_{(i,j)}:= \lfloor (q^3(M-j) + qM(q-i) - p^b (q+1)M)/p^b (q+1)M \rfloor$ be as above. Then
$$
\displaystyle |\Lambda(\negP_{m+1})| = \dfrac{(2m-1)!}{(m-1)! m!} + \sum_{i=0}^{q} \sum_{\substack{j=1 \\ \tau_{(i,j)}\geq 0}}^{M} \dfrac{(\tau_{(i,j)} + m )!}{\tau_{(i,j)}! m!}.
$$
\end{proposition}

\begin{proof}
Let $\mathbb{A}:= \{  \negdelta_{k_2,\ldots,k_{m+1}}^{i,j} \in \hLambda(\negP_{m+1}) \mbox{ ; }  k_2, \ldots, k_{m+1} \in \mathbb{N}_{0} \mbox{ with }  \sum_{\ell=2}^{m+1} k_{\ell} \leq \tau_{(i,j)}\}$
and $\mathbb{B}:=\{ \neglambda_{{k_2}, \ldots , k_{m+1}} \in \mathbb{N}_{0}^{m+1} \mbox{ ; }k_2,\ldots,k_{m+1} \in \mathbb{N}_{0} \mbox{ such that } \sum_{\ell=2}^{m+1} k_{\ell} \leq m-1 \}.$
Thus, by Corollary \ref{corolario lambda}, $\Lambda(\negP_{m+1}) = \mathbb{A} \cup \mathbb{B}$. For each $\ell \in \{2,\ldots,m+1\}$, $k_{\ell}(q+1)M +iM+j \geq 0$ if and only if $k_{\ell}\geq 0$. Now, note that if $\tau_{(i,j)}<0$, then $\negdelta_{k_2,\ldots,k_{m+1}}^{i,j} \notin \mathbb{A}$, since there are not $k_2, \ldots, k_{m+1} \in \mathbb{N}_{0} \mbox{ with }  \sum_{\ell=2}^{m+1} k_{\ell} \leq \tau_{(i,j)}$. So,
using Lemma \ref{delta lambda} and the number of non-negative integer solutions ${k_2}, \ldots , k_{m+1}$ to $\sum_{\ell=2}^{m+1} k_{\ell} \leq \tau_{(i,j)}$ and $\sum_{\ell=2}^{m+1} k_{\ell} \leq m-1$ we can conclude that $\displaystyle |\mathbb{A}|=\sum_{i=0}^{q} \sum_{\substack{j=1 \\ \tau_{(i,j)}\geq 0}}^{M} \dfrac{(\tau_{(i,j)} + m )!}{\tau_{(i,j)}! m!}$ and $|\mathbb{B}| = \dfrac{(2m-1)!}{(m-1)! m!}$. Now, by Lemma \ref{delta lambda}, we have that $\mathbb{A} \cap \mathbb{B} = \emptyset$ and the result follows.
\end{proof}

By Theorem \ref{teo gaps}, we have that $\displaystyle G(\negP_{m+1})=\bigcup_{\negbeta \in\Lambda(\negP_{m+1})} (\overline{\nabla}(\negbeta)\cap \NN_0^{m+1})$. In the follow we present a bound for the cardinality of $G(\negP_{m+1})$.

\begin{proposition}
Let $\Lambda_{m+1}:=|\Lambda(\negP_{m+1})|$ and suppose that $\Lambda(\negP_{m+1})=\{ \negbeta_1, \ldots, \negbeta_{\Lambda} \}$, with $\negbeta_k = (\beta_{1}^{(k)}, \ldots , \beta_{m+1}^{(k)})$ for each $k \in \{ 1, \ldots , \Lambda\}$. Then
$$
\displaystyle |G(\negP_{m+1})| \leq \sum_{k=1}^{\Lambda_{m+1}} \left( \sum_{r=1}^{m+1} \prod_{s\neq r} \beta_{s}^{(k)} \right).
$$
\end{proposition}

\begin{proof}
For each $\negbeta_k \in \Lambda(\negP_{m+1})$, we have that $\displaystyle |\overline{\nabla}(\negbeta_k)\cap \NN_0^{m+1}| = \sum_{r=1}^{m+1} \prod_{s\neq r} \beta_{s}^{(k)}$ and the result follows.  
\end{proof}

For the particular case $m=1$ we have that 
$$
\displaystyle \Lambda_2= |\Lambda(\negP_{2})| = 1 + \sum_{i=0}^{q} \sum_{\substack{j=1 \\ \tau_{(i,j)}\geq 0}}^{M} \tau_{(i,j)}.
$$

Given $(m_1,m_2), (n_1,n_2) \in \mathbb{N}_{0}^{2}$, we write $(m_1, m_2) \prec (n_1,n_2)$ if $m_2 < n_2$.

\begin{remark}
Note that, by Lemma \ref{delta lambda} and since $(i,j) \in ([0,q]\times[1,M]\cap \ZZ^2)\backslash (q,M)$, for all $\negalpha=(\alpha_1,\alpha_2), \negbeta=(\beta_1,\beta_2) \in \Lambda(\negP_{2})$, with $\negalpha \neq \negbeta$, we have that $\alpha_1 \neq \beta_1$ and $\alpha_2 \neq \beta_2$. So, we have $\negalpha \prec \negbeta$ or $\negbeta \prec \negalpha$. 
\end{remark}

Let $n \geq 1$ be an integer and let $A=\{ \negalpha_1 \prec \negalpha_2 \prec \cdots \prec \negalpha_n \} \subset \mathbb{N}_{0}^2$, where $\negalpha_k = (\alpha_{1}^{(k)}, \alpha_{2}^{(k)})$, for each $k=1,\ldots,n$. Define $\zeta_1(A):=0$ and, for $t=2,\ldots,n$,
\begin{equation}\label{zeta}
\zeta_t(A):=|\{ \negalpha_k =(\alpha_{1}^{(k)}, \alpha_{2}^{(k)}) \in A \mbox{ ; }  \negalpha_k \prec \negalpha_t \mbox{ and } \alpha_{1}^{(k)} > \alpha_{1}^{(t)}\}|.
\end{equation}

\begin{proposition}
Let $\Lambda(\negP_{2})=\{ \negbeta_1 \prec \negbeta_2 \prec \cdots \prec \negbeta_{\Lambda_2} \}$, where, for each $t=1,\ldots,\Lambda_2$, $\negbeta_t = (\beta_{1}^{(t)}, \beta_{2}^{(t)})$, and $\zeta_t(\Lambda(\negP_{2}))$ is given (\ref{zeta}). Then
$$
\displaystyle |G(\negP_2)| = \sum_{t=1}^{\Lambda_2}[ \beta_{1}^{(t)} + \beta_{2}^{(t)} - \zeta_t(\Lambda(\negP_{2}))].
$$
\end{proposition}
\begin{proof}
The results follows directly from Theorem \ref{teo gaps} and the definition of $\zeta_t(\Lambda(\negP_{2}))$.
\end{proof}

\section{Generalized Weierstrass Semigroup at certain $m+1$ points on $\mathcal{Y}_{n,s}$}

 Denoting for simplicity the point $P_{(\alpha_i,0,0)} \in \mathcal{Y}_{n,s}$ as given in Equation (\ref{div z Yns}) by $P_{i}$, where $1 \leq i \leq q$, and taking $p^b=1$ in the equations, results and proofs in the previous section, we get the similar results for the absolute and relative maximals elements in $\widehat{H} (\mathbf{P}_{m+1})$ and the minimal generating set of the Weierstrass semigroup $H (\mathbf{P}_{m+1})$, for $1 \leq m \leq q$, where $\mathbf{P}_{m+1} = (P_{\infty}, P_1, \ldots , P_m)$. We will summarize the main results below. The proofs will be omitted because they are analogous to those presented in the previous section.


\begin{theorem} Let $1\leq m\leq q$ and $\negP_{m+1} = (P_{\infty}, P_1, \ldots , P_{m})$. Let
$$\negalpha^{i,j,m}:=\left((q^2 - m)(q+1)M -iqM -jq^3,iM+j,\ldots,iM+j\right)\in \ZZ^{m+1}.$$
Then
$$\hGamma(\negP_{m+1})\cap \cC_{m+1}=\{\negalpha^{i,j,m} \ : (i,j)\in ([0,q]\times[1,M]\cap \ZZ^2) \backslash (q,M)\}\cup \{{\bf 0}\}.$$
\end{theorem}

\bigskip

\begin{corollary} \label{corolario gamma Y}
$$
\begin{array}{ll}
\hGamma(\negP_{m+1})= & \left\{  \left((q^2 - m - \sum_{\ell=2}^{m+1} k_\ell)(q+1)M -iqM -jq^3, \right. \right. \\
& k_2(q+1)M + iM+j,\ldots,k_{m+1}(q+1)M + iM+j ) \in \ZZ^{m+1} \mbox{ ; } \\
& (i,j)\in ([0,q]\times[1,M]\cap \ZZ^2) \backslash (q,M) \mbox{ , }k_\ell \in \ZZ  \mbox{ for } \ell=2,\ldots,m+1 \}.
\end{array}
$$
\end{corollary}

\bigskip

For $(i,j)\in ([0,q]\times[1,M]\cap \ZZ^2) \backslash (q,M)$ and $k_2, \ldots,k_{m+1} \in \mathbb{Z}$, let
\begin{equation} \label{beta classico}
\begin{array}{ll}
\negbeta_{k_2,\ldots,k_{m+1}}^{i,j,m}:= & \left((q^2 - m - \sum_{\ell=2}^{m+1} k_\ell)(q+1)M -iqM -jq^3,\right . \\
& \left . k_2(q+1)M + iM+j,\ldots,k_{m+1}(q+1)M + iM+j \right) \in \ZZ^{m+1}.
\end{array}
\end{equation}

\begin{corollary} \label{classical WS Y}
Let $i,j,k_2, \ldots, k_{m+1}$ and $\negbeta_{k_2,\ldots,k_{m+1}}^{i,j,m}$ be as above. Then, $\Gamma(\negP_{m+1})$ is the set of all $\negbeta_{k_2,\ldots,k_{m+1}}^{i,j,m} \in \mathbb{N}_{0}^{m+1}$.
\end{corollary}

\bigskip

\begin{theorem} Let $1\leq m\leq q$ and $\negP_{m+1} = (P_{\infty}, P_1, \ldots , P_{m})$. Let
$$\negbeta^{0,0,m}=\left((m-1)(q+1)M,0,\ldots,0\right)\in \ZZ^{m+1}$$
and for $(i,j)\in ([0,q]\times[1,M]\cap \ZZ^2) \backslash (q,M)$, let
$$\negbeta^{i,j,m}:=\left((q^2 - 1)(q+1)M -iqM -jq^3,iM+j,\ldots,iM+j\right)\in \ZZ^{m+1}.$$
Then
$$\hLambda(\negP_{m+1})\cap \cC(\negP_{m+1})=\{\negbeta^{i,j,m} \ : (i,j)\in ([0,q]\times[1,M]\cap \ZZ^2) \backslash (q,M)\}\cup \{\negbeta^{0,0,m}\}.$$
\end{theorem}

\bigskip

\begin{corollary} \label{corolario lambda chapeu Y}
$$
\begin{array}{ll}
\hLambda(\negP_{m+1})= & \left\{  \left((q^2 - 1-\sum_{\ell=2}^{m+1} k_\ell)(q+1)M -iqM -jq^3, \right. \right. \\
& k_2(q^n+1)M + iM+j,\ldots,k_{m+1}(q^n+1)M + iM+j ) \in \ZZ^{m+1} \mbox{ ; } \\
& (i,j)\in ([0,q]\times[1,M]\cap \ZZ^2) \backslash (q,M) \mbox{ , }k_\ell \in \ZZ  \mbox{ for } \ell=2,\ldots,m+1 \}\\
& \bigcup \left\{ (m-1 - \sum_{\ell=2}^{m+1} \widetilde{k}_\ell)(q+1)M,\widetilde{k}_2(q^n+1)M,\ldots,\widetilde{k}_{m+1}(q^n+1)M) \mbox{ ; } \widetilde{k}_\ell \in \ZZ \right\}.
\end{array}
$$
\end{corollary}

\bigskip

\begin{corollary} \label{corolario lambda Y}
$$
\begin{array}{rl}
\Lambda(\negP_{m+1})= & \left\{  \left((q^2 - 1-\sum_{\ell=2}^{m+1} k_\ell)(q+1)M -iqM -jq^3, \right. \right. \\
& k_2(q^n+1)M + iM+j,\ldots,k_{m+1}(q^n+1)M + iM+j ) \in \ZZ^{m+1} \mbox{ ; } \\
& (i,j)\in ([0,q]\times[1,M]\cap \ZZ^2) \backslash (q,M) \mbox{ , }k_\ell \in \ZZ \mbox{ and } k_\ell(q^n+1)M + iM+j \geq 0\\
& \mbox{ for } \ell=2,\ldots,m+1 \mbox{ ,and } (q^2 - 1-\sum_{\ell=2}^{m+1} k_\ell)(q+1)M -iqM -jq^3 \geq 0 \}\\
\bigcup & \{ (m-1 - \sum_{\ell=2}^{m+1} \widetilde{k}_\ell)(q+1)M,\widetilde{k}_2(q^n+1)M,\ldots,\widetilde{k}_{m+1}(q^n+1)M) \mbox{ ; } \widetilde{k}_\ell \in \mathbb{N}_0 \mbox{ and } \\
& \sum_{\ell=2}^{m+1} \widetilde{k}_\ell \leq m-1\}.
\end{array}
$$
\end{corollary}

\end{document}